\documentclass{amsart}
\usepackage{amsmath}
\usepackage{amsfonts}
\usepackage{amssymb}
\usepackage{graphicx}

 \newtheorem{thm}{Theorem}[section]
 \newtheorem{cor}[thm]{Corollary}
 \newtheorem{lem}[thm]{Lemma}
 \newtheorem{prop}[thm]{Proposition}

 \theoremstyle{remark}
 \newtheorem{rem}[thm]{Remark}

\newtheorem*{ack}{Acknowledgement}
  \newtheorem{ex}[thm]{Example}
 \numberwithin{equation}{section}
\begin{document}
\title[Minimal non-nilpotent groups which are supersolvable]{Minimal non-nilpotent groups which are supersolvable}
\author[F. G. Russo]{Francesco G. Russo}
\address{Mathematics Department\\
University of Naples Federico II\\
via Cinthia, 80126, Naples, Italy}
\email{francesco.russo@dma.unina.it}
%\author[N. Trabelsi]{Nadir Trabelsi}
%\address{Mathematics Department\\
%University Ferhat Abbas\\
%19000, Setif, Algeria} \email{nadir\_trabelsi@yahoo.fr}
\date{\today}



\begin{abstract} The structure of a group which is not nilpotent but all of whose proper subgroups are nilpotent has interested the
researches of several authors both in the finite case  and in the
infinite case. The present paper generalizes  some classic
descriptions of M. Newman, H. Smith and J. Wiegold in the context of
supersolvable groups.
\end{abstract}

\maketitle

\section{Introduction}

Let $\mathfrak{N}$ be the class of all nilpotent groups. A group $G$
is said to be a $minimal$ $non$-$nilpotent$ $group$, or
$\mathfrak{N}$-$critical$ $group$, or $Schmidt$ $group$, or
$MNN$-$group$, if it doesn't belong to $\mathfrak{N}$ but all of
whose proper subgroups belong to $\mathfrak{N}$. We will use the
last terminology in the present paper. It is evident already from
these 4
 ways to call the same mathematical object that there is a
wide literature on the topic. Many authors are still interested in
studying $MNN$-groups, because they play an important role from the
point of view of the general theory. The first example of finite
$MNN$-group is probably the symmetric group $S_3$ of order 6. We
know that $S_3$ can be written as the semidirect product of a cyclic
group $C_3$ of order 3 by a cyclic group $C_2$ of order $2$, which
acts by inversion on $C_3$. Already for $S_3$  the condition of
being an $MNN$-group determines its structure, in fact, we have a
semidirect product and this allows us to have a deep knowledge of
the whole group. At this point the following question becomes
natural.

\medskip

What is the influence of being an $MNN$-group on the group
structure?

\medskip

In the finite case a first answer is due to a famous contribution of
O. Yu. Schmidt and more details can be found in \cite{shem}. His
methods and techniques showed that the question can be seen from a
different prospective, involving the theory of classes of groups and
conditions which are weaker of being nilpotent. A recent
contribution in this direction has been given by J.C.Beidleman and
H.Heineken in \cite[Theorem 2]{beidleman-heineken}, where they
generalize the description of O. Yu. Schmidt to the context of
saturated formation of finite groups.

On another hand,  classic descriptions of $MNN$-groups in the
infinite case were given by M. Newman, H. Smith and J. Wiegold in
\cite{NW, smith1,smith2}. Among these groups, there are Tarski
groups \cite{Ols} so it is a common use the imposition of suitable
finiteness conditions in order to treat separately the Tarski
groups.

Now we illustrate the new idea of the present paper. Consider the
following subset of the subgroup lattice $\mathcal{L}(G)$ of $G$
\begin{equation}\label{1}
\mathcal{M}(G)=\{H\leq G: H\not\in \mathfrak{N}\}.
\end{equation}
$\mathcal{M}(G)=\{G\}$ if and only if $|\mathcal{M}(G)|=1$, that is,
$G$ is the unique non-nilpotent subgroup, that is, $G$ is an
$MNN$-group. It turns out that we may extend significatively the
classifications in \cite{NW, shem, smith1}, dealing with (\ref{1})
when $|\mathcal{M}(G)|=m\geq1$. For the case $m=2$ we can be more
precise and details are illustrated in Section 2, preparing the main
results which are in Section 3. For higher values of $m$ we have not
found deep restrictions on the group structure and, to the best of
our knowledge, it is an open problem.

\section{The Case $m=2$}
The motivation of studying (\ref{1}) is clear once we note that
$|\mathcal{M}(G)|$ gives a measure of how many $MNN$-subgroups are
contained in $G$, and so , of how $G$ is far from the usual
classifications in \cite{NW, shem, smith1}. Of course,
$|\mathcal{M}(G)|=2$ if and only if $G\not \in \mathfrak{N}$ and we
have just 1 $MNN$-subgroup $K$ of $G$. Going on, the situation is a
little bit more complicated. Already the case $|\mathcal{M}(G)|=3$
needs of more attention.

\begin{lem}\label{l:1}
$|\mathcal{M}(G)|=2$ if and only if $G\not \in \mathfrak{N}$ and $G$ contains
a maximal normal subgroup $K$ which is an $MNN$-group.
\end{lem}

\begin{proof}
Since $|\mathcal{M}(G)|=2$, we have $\mathcal{M}(G)=\{G,K\}$, where
$K<G$. So $K$ is an $MNN$-group. If there is a subgroup $H$ of $G$
such that $K<H<G$, then $H\in \mathfrak{N}$ and so $K\in
\mathfrak{N}$. This contradiction implies that $K$ is a maximal
subgroup of $G$. Now for each $x\in G$, $K^x\leq G$. But $K^x\simeq
K\not \in \mathfrak{N}$, so $K^x=K$. Then $K$ is normal in $G$.
\end{proof}

\begin{lem}\label{l:2}
Assume $|\mathcal{M}(G)|=2$ and $K$ as in Lemma \ref{l:1}. Then
$G/K$ is of prime order and $G'\leq K$.
\end{lem}

\begin{proof}
Since $G/K$ has only two subgroups, $G/K$ is of prime order. Since
$G/K$ is abelian, $G'\leq K$.
\end{proof}

\begin{rem}\label{r:1}
Assume $|\mathcal{M}(G)|=2$. Then $K$ in Lemma \ref{l:1} is a
characteristic subgroup of $G$.
\end{rem}

\begin{proof}
Let $\alpha\in $Aut$(G)$. Then $\alpha(K)\simeq K\not\in
\mathfrak{N} $, so $\alpha(K)=K$.
\end{proof}

In order to proceed we recall the Hall's Criterion of nilpotence in
\cite[5.2.10]{GT}.

\begin{thm}[P.Hall, see \cite{GT}] \label{t:Hall}Let $N$ be a normal subgroup of
a group $G$. If $N\in \mathfrak{N}$ and $G/N'\in \mathfrak{N}$, then
$G\in \mathfrak{N}$.
\end{thm}

\begin{rem}\label{r:2}
Assume $|\mathcal{M}(G)|=2$ and $G'<K$ with $K$ as in Lemma
\ref{l:1}. Then $G$ is solvable with a non-trivial non-nilpotent
homomorphic image.
\end{rem}

\begin{proof}
Since $G'<K$ and $K$ is an $MNN$-group, $G'\in \mathfrak{N}$ and so
$G$ is solvable. Theorem \ref{t:Hall} implies $G/G''\not\in
\mathfrak{N}$, which is the requested homomorphic image.
\end{proof}

\begin{rem}\label{r:3}
Let $K$ be as in Remark \ref{r:2}. If $\mathcal{M}(G)=\{G,K\}$, then
$\mathcal{M}(G/G'')=\{G/G'',K/G''\}$.
\end{rem}

\begin{proof}
Remark \ref{r:2} shows that $G/G''\not \in \mathfrak{N}$. Each
subgroup of $G/G''$ is of the form $H/G''$, where $G''\leq H \leq
G$. Then $H/G''\in \mathfrak{N}$, whenever $H\not=K$ and $H\not=G$.
Therefore, $K/G''\not\in \mathfrak{N}$ by Theorem \ref{t:Hall}.
\end{proof}

A group $G$ is $locally$ $graded$ if every nontrivial finitely
generated subgroup of $G$ have a finite image. The next result
recalls \cite[Theorem 2]{smith2}.

\begin{thm}[H. Smith, see \cite{smith2}] \label{t:Smith}Let $G$ be a
locally graded group and suppose that, for some positive integer
$b(G)$, every non-nilpotent subgroup of $G$ is subnormal of defect
$\leq b(G)$ in $G$. Then $G$ is solvable.
\end{thm}

Now Remark \ref{r:2} can be reformulated in the following way.

\begin{prop}\label{p:1}
Assume $|\mathcal{M}(G)|=2$. If $G$ is locally graded, then $G$ is
solvable.
\end{prop}

\begin{proof}
Let $K$ be as in Lemma \ref{l:1}. All non-nilpotent subgroups of $G$
are subnormal. Then $G$ is solvable by Theorem \ref{t:Smith}.
 \end{proof}

\begin{lem}\label{l:3}
Assume $|\mathcal{M}(G)|=2$ and $K$ as in Remark \ref{r:2}. If
$M\not=K$ is a maximal normal subgroup of $G$, then $[K,M]\not=1$.
\end{lem}

\begin{proof}
Assume $[K,M]=1$. Then $M\leq C_G(K)$. If $M=C_G(K)$, then $M\cap
K=Z(K)$ and so $MK/M\simeq K/(M\cap K)\simeq K/Z(K)$ is cyclic. This
gives $K$ abelian. If $G=C_G(K)$, then $K\leq Z(G)$ and again $K$ is
abelian. Both cases contradict $K\not\in \mathfrak{N}$. The result
follows.
\end{proof}

\begin{prop}\label{p:2}
Assume $|\mathcal{M}(G)|=2$ and $K$ as in Remark \ref{r:2}. If $M$
is a maximal subgroup of $G$ whose elements have coprime order with
those of $K$, then $K$ is the unique maximal subgroup of $G$.
\end{prop}
\begin{proof}
$K$ is periodic by the classification of M.Newman and J.Wiegold in
\cite{NW}. Then $G$ is periodic and so $M$. $M\cap K=\{1\}$ from the
relation $(|\langle m\rangle|, |\langle k \rangle|)=1$ for each
$m\in M$ and $k\in K$. Then, $[M,K]\leq M\cap K=\{1\}$. Lemma
\ref{l:3} implies that $M=K$ and the result follows.
\end{proof}

Recall that $\pi(G)$ denotes the set of prime divisors of the order  of the elements of $G$.

\begin{cor}\label{c:2}
Assume $|\mathcal{M}(G)|=2$ and $K$ non-finitely generated. If $K$
has maximal subgroups, then $G$ is a  Chernikov group of derived
length at most 3 with $|\pi(G)|\leq2$.
\end{cor}
\begin{proof}
By the classification of M.Newman and J.Wiegold in \cite{NW},  $K$
is a metabelian Chernikov $p$-group for some prime $p$ (see
\cite[p.242, lines +5 and +6]{NW}).  From Lemma \ref{l:2} the result
follows easily.
\end{proof}

\begin{lem}\label{l:4}
Assume $|\mathcal{M}(G)|=2$ and $K$ non-finitely generated. Then $G$
is locally nilpotent. In particular, each maximal subgroup of $G$ is
normal and of prime index.
\end{lem}

\begin{proof}
Every finitely generated subgroup $H$ of $G$, such that $H\not=G$ and $H\not=K$, is nilpotent. Then $G$ is locally nilpotent. The remaining part of the result follows easily.
\end{proof}

\begin{cor}\label{c:3}
Assume $|\mathcal{M}(G)|=2$ and $K$ non-finitely generated. Then $G$
is solvable.
\end{cor}

\begin{proof}
This follows from Lemma \ref{l:4} and Proposition \ref{p:1}.
\end{proof}

A concrete situation is described as follows.

\begin{ex}\label{example:1} Write $A=C_{2^{\infty}}$ for the quasicyclic $2$-group, $B=\langle x \rangle$ and $C=\langle
y \rangle$, where $x$ and $y$ have order 2. Consider $K=A\rtimes B$,
which is the well-known locally dihedral 2-group \cite[p.344]{GT},
and  $G=K\times C$.  By construction, $\mathcal{L}(K)-
\mathcal{L}(A)=\{K, B, \langle H,B\rangle \}$, where
$\{1\}\not=H<A$. Of course, $B\in \mathfrak{N}$. On  the other hand,
$\langle H,B \rangle \leq Z_i(K)$ for some $i\geq1$, since $K$ is
$\omega$-hypercentral. Then $\langle H,B \rangle\in \mathfrak{N}$.
We conclude that $K$ is an $MNN$-group. Now, the presence of $K$
implies that $G$ is not an $MNN$-group. By construction,
$\mathcal{L}(G)-\mathcal{L}(K)=\{G, C, \langle L,C\rangle\}$, where
$\{1\}\not=L<K$. Noting that $\langle L, C\rangle=L\times C$, we
have $L\times C\in \mathfrak{N}$. Then $\mathcal{M}(G)=\{G,K\}$.
Note that $K$ is the unique maximal subgroup of $G$. Note also that
$A$ is the unique maximal subgroup of $K$. We have all it is needed
in order to state that $G$ satisfies Proposition \ref{p:1} and
Corollaries \ref{c:2}, \ref{c:3}. \end{ex}

Example \ref{example:1} shows that we may get groups as in
Proposition \ref{p:1} and Corollaries \ref{c:2}, \ref{c:3}, adding a
finite cyclic group to a given $MNN$-group. Then, choosing a
suitable order for the cyclic group, we may give examples for
Proposition \ref{p:2}.

\section{Main Theorems}

In order to proceed with the proof of the main theorem of the
present section, we recall \cite[Lemma 3.2]{NW} and \cite[Theorem
2.12]{NW}, respectively.

\begin{lem}[M.Newman--J.Wiegold, see \cite{NW}]\label{l:NW} Let $G$
be a finitely generated non-nilpotent group all of whose proper
subgroups are locally nilpotent and $\gamma_\infty(G)$ be the last
term of the lower central series of $G$. If $G/\gamma_\infty(G)$ is
nontrivial, then $G$ is finite.
\end{lem}

\begin{thm}[M.Newman--J.Wiegold, see \cite{NW}]\label{t:NW} If $G$
is a group in which every pair of proper normal subgroups generates
a proper subgroup, then $G/G'$ is a locally cyclic $p$-group for
some prime $p$ and $G'=\gamma_\infty(G)$.
\end{thm}

We should recall also some notations from \cite{Ved}. Let $n\geq1$,
$i$ and $j$ be two distinct integers in $\{1,2,\dots,n\}$, $p_i,p_j$
primes, $d_i,d_j\geq1$, $\pi(d_i)$ be the set of prime divisors of
$d_i$ and $q_i\in \pi(d_i)$. An $F_{\{p_i,d_i\}}$-$group$ is a
Frobenius group whose kernel is an elementary abelian group of order
$p_i^{m_i}$ with cyclic complement of order $d_i$, where $m_i$ is
the exponent of $p_i$ modulo $q_i$. The next result quotes
\cite[Theorem 1]{Ved}.

\begin{thm} \label{t:Ved}
In a non-nilpotent finite group $G$, all $MNN$-subgroups are
subnormal if and only if \begin{equation}G/Z_{\infty}(G)=G_1\times
G_2\times\ldots \times G_n,\end{equation} where $G_i$ is an
$F_{\{p_i,d_i\}}$-group,  and $(d_i,d_j)=1$ for any $i\not=j$ with
$i,j\in \{1,2,\ldots,n\}$.
\end{thm}

Our main result is the following and describes (\ref{1}) in a
special case.

\begin{thm} \label{t:1}
Assume $K$ as in Remark \ref{r:2}. If $K$ is finitely generated,
then $G$ is a finite supersolvable group. Furthermore,
\begin{equation}G/Z_{\infty}(G)=G_1\times G_2\times\ldots \times G_n,\end{equation} where
$G_i$ is an $F_{\{p_i,d_i\}}$-group and $(d_i,d_j)=1$ for any
$i\not=j$ with $i,j\in \{1,2,\ldots,n\}$.
\end{thm}

\begin{proof}
An application of Lemma \ref{l:NW} to $K$ implies that $K$ is
finite. Then $G$ is finite by Lemma \ref{l:2}. More precisely,
$G=K\langle x\rangle$, where $|\langle x\rangle|=|G/K|=q$ for some
prime $q$. By Theorem \ref{t:NW} we may deduce that $|K/K'|$ is a
cyclic group of order $p^r$ for some prime $p$ and some $r\geq1$.
Then $K=K'\langle y\rangle$, where $|\langle y \rangle|=p^r$, and so
$G=\langle K', x, y\rangle=K' \langle x, y \rangle$, where $K'$ is
nilpotent finitely generated of class $c$. We know from
\cite[5.2.18]{GT} that a finitely generated nilpotent group has a
central series whose factors are cyclic with prime or infinite
orders and so $K'=S$ is supersolvable and we have the following
series $\{1\}=Z_0(S)\triangleleft Z_1(S) \triangleleft \ldots
\triangleleft Z_c(S)=S \triangleleft K \triangleleft G,$ where
$Z_1(S)/Z_0(S), \ldots, Z_c(S)/Z_{c-1}(S)$ are cyclic groups of
prime order. We have just seen that $K/S$ is a cyclic group. $G/K$
is cyclic by Lemma \ref{l:2}. Note that each term of this series is
normal in $G$. Therefore $G$ is supersolvable.

Independently, the only fact that $G$ is finite allows us to apply
\cite[Theorem 1]{Ved} and so $G/Z_{\infty}(G)$ is the direct product
of Frobenius groups as claimed.
\end{proof}

It is interesting the following consequence of Theorem \ref{t:1}.

\begin{cor}\label{c:4}
If $G$ is a finite solvable group with $|\mathcal{M}(G)|=2$, then it
is supersolvable.
\end{cor}

\begin{rem}\label{r:4}
Theorem \ref{t:1} relates $G/Z_\infty(G)$ with $|\mathcal{M}(G)|$.
Recall that nilpotent finitely generated groups are supersolvable
(see \cite{GT}). Then we are saying in Theorem \ref{t:1} that small
values of $|\mathcal{M}(G)|$ imply that $G$ is a (finite)
supersolvable group which is not nilpotent. Furthermore we are
describing, thanks to $G/Z_\infty(G)$, how much is big the
difference from being supersolvable and not being nilpotent.
\end{rem}

The remainder of this section illustrates another aspect of
(\ref{1}).

We recall from \cite[\S 13.3]{GT} that
\begin{equation}\omega(G)={\underset{S sn G} \bigcap N_G(S)} \end{equation} is the
$Wielandt$ $subgroup$ of a group $G$. $\omega(G)$ is always a
$T$-$group$, that is, a group in which the normality is a transitive
relation. Solvable $T$-groups were classified by D.J.Robinson in
1964 (see \cite{lr}) and more generally  the groups in which all the
subgroups are subnormal were classified by W. M\"ohres in
\cite{moehres2} (see also \cite[\S12.2]{lr}). These are related to
$MNN$-groups by \cite[Theorem 3.1]{smith1}, which is quoted below.

\begin{thm}[H.Smith, see \cite{smith1}]\label{t:S} Let $G$ be a
soluble $MNN$-group and suppose that $G$ has no maximal subgroups.
Then:
\begin{itemize}
\item[(i)] $G$ is a countable $p$-group for some prime $p$ and
$G/G'\simeq C_{p^\infty}$;
\item[(ii)]every subgroup of $G$ is subnormal;
\item[(iii)]every hypercentral image of $G$ is abelian and
$G'=\gamma_\infty(G)$;
\item[(iv)] every radicable subgroup of $G$ is central;
\item[(v)]$HG'=G$ implies $H=G$ for every subgroup $H$ of $G$ and
$C_G(G')$ is abelian. In particular, $G$ has no proper subgroups of
finite index;
\item[(vi)] $G'$ is not the normal closure in $G$ of a finite
subgroup;
\item[(vii)] $Z(G)=Z_\infty(G)$.
\end{itemize}
\end{thm}

We have all it is necessary in order to prove the second main result
of this section.

\begin{thm}\label{t:2}
Assume $|\mathcal{M}(G)|=2$ and $K$ non-finitely generated as in
Lemma \ref{l:1}. If $K$ has no maximal subgroups, then $K/\omega(K)$
is non-trivial, non-abelian, of infinite exponent and at least of
countable abelian rank.
\end{thm}

\begin{proof}
From Corollary \ref{c:3}, $G$ is solvable. Assume $\omega(K)=K$. $K$
is a $T$-group and Theorem \ref{t:S} (ii) every subgroup of $K$ is
subnormal. Both these conditions imply that $K$ is a Dedekind group,
then $K\in \mathfrak{N}$. This  contradiction shows that
$K/\omega(K)$ is non-trivial.

Assume $[K/\omega(K), K/\omega(K)]=1$. Then
\begin{equation}[K,K]\leq \omega (K)={\underset{S sn K} \bigcap
N_K(S)}= {\underset{S \leq K} \bigcap N_K(S)}=norm(K)\leq
Z_2(K),\end{equation} where the last inequality is due to a famous
result of E. Schenkman \cite{schenk}. Therefore $K\in \mathfrak{N}$,
which is a contradiction. This implies that $K/\omega(K)$ cannot be
abelian.

The  fact that $K/\omega(K)$ is of infinite exponent follows by the
classification of W.M\"ohres and precisely by
\cite[Theorem]{moehres2}.

Note that $K/\omega(K)$ has no maximal subgroups. Then $K/\omega
(K)$ has no proper subgroups of finite index. On another hand , we
know from \cite[5.3.6]{lr} that a solvable group with finite abelian
rank  and  no proper subgroups of finite index must be nilpotent.
This implies that $K/\omega(K)$ cannot be of finite abelian rank,
and so, at least of countable abelian rank.
\end{proof}

Unfortunately, we cannot think to Example \ref{example:1} in case of
Theorem \ref{t:2}, since in Example \ref{example:1} there are
maximal subgroups. However, a satisfactory description is offered by
the following result.

\begin{cor}\label{c:5}
Assume $|\mathcal{M}(G)|=2$ and $K$ non-finitely generated as in
Lemma \ref{l:1}. If $K$ has no maximal subgroups, then $K$ has the
series \begin{equation}\{1\}\triangleleft \omega(G)=K^{(d)}
\triangleleft K^{(d-1)} \triangleleft \ldots \triangleleft K'
\triangleleft K \triangleleft G,\end{equation} where
$\omega(K)=\gamma_3(\omega(K))\rtimes L$, $L$ is the subgroup
generated by the involutions of $\omega(K)$, $K/K'\simeq
C_{p^{\infty}}$ for some prime $p$, there exists some
$i\in\{1,\ldots,d\}$ such that $K^{(i+1)}/K^i$ is the direct product
of infinitely many copies of $C_{p^{\infty}}$, $G/K$ is of prime
order.
\end{cor}

\begin{proof}
$G$ is solvable by Corollary \ref{c:3}. $K$ is a solvable
$MNN$-group with no maximal subgroups and it must be a periodic
$p$-group, by Theorem \ref{t:S}. The fact that $\omega(K)$ is a
semidirect product of $L$ and $\gamma_3(\omega(K))$ follows from the
classification of periodic solvable $T$-groups and can  be found for
instance in \cite[Exercises 13.4, n.10, p.394]{GT}. Now the rest of
the result follows from the combination of Lemma \ref{l:1}, Theorem
\ref{t:2}, \cite[Exercises 13.4, n.10, p.394]{GT} and Theorem
\ref{t:S}.
\end{proof}

\begin{ack}
The author is grateful to the Monastero di S.Pasquale a  Chiaja of
Naples for hospitality  in the period in which the present paper was
written.
\end{ack}

\end{document}